\documentclass[a4paper,12pt]{article}
\usepackage{amsfonts}
\usepackage{latexsym}
\usepackage{amsthm}
\usepackage{amssymb}
\usepackage{amsmath}
\usepackage{eufrak}
\usepackage{mathrsfs}
\usepackage{geometry}
\usepackage{graphics}
\usepackage{comment}
\usepackage[all]{xy}
\usepackage{rotating}

\newtheorem{theo}{Theorem}[section]
\newtheorem{definition}[theo]{Definition}
\newtheorem{Lemma}[theo]{Lemma}

\newtheorem{Prop}[theo]{Proposition}

\newtheorem*{Rm}{Remark}

\newtheorem{Ex}{Example}

\newtheorem*{tma}{Theorem \ref{d par}}

\newcommand{\lra}{ \longrightarrow }

\renewcommand{\Im}{{\rm Im}}

\author{Natalia A. Viana Bedoya. \ and Daciberg Lima Gon\c calves} 
\title{Indecomposability of branched coverings of even degree on the projective plane\footnote{ This work was partially supported 
by Projeto Tem\'atico Topologia Alg\'ebrica, Geom\'etrica
e Differencial-
2008/57607-6. }}
\date{}
\begin{document} 
\maketitle

\begin{abstract} In this work we characterize  branch data of branched coverings of even 
degree over the projective plane  which are realizable by indecomposable branched coverings. 
\end{abstract}

\section{Introduction}
 It is a natural and standard question to know whether   a map in a given  class of maps  is a 
composition or not of maps in the same
 class. In 1957  Borsuk and Molski \cite{BM} asked about the existence of a continuous map of finite
order\footnote{A continuous map $\phi$ defined on a space $X$ is said to be of \emph{order} $\leq k \in
\mathbb{Z}^{+}$ if for any $y \in \phi(X)$, $\phi^{-1}(y)$ contains at most
$k$ points.}, which is not a composition of \emph{simple maps} (maps of order $\leq 2$). For the historical
development of this question  see  \cite{BBZ}   and \cite{BN&GDL}.   More recently, in 
 2002, Krzempek  \cite{KJ} constructed covering maps on locally arcwise connected continua that are not factorizable 
into covering maps of order $\leq n-1$, for all $n$.  Also in 2002,
 Bogataya, Bogaty{\u\i} and Zieschang \cite{BBZ} 
gave an example  of a 4-fold covering of a surface of genus 2 by a surface of genus 5
that cannot be represented as a composition of two non-trivial open maps.  
The relevant and  general problem of   classifying the branched coverings from the    viewpoint  of  decomposability
 is related with the Inverse Galois problem (see  for example the references 
\cite{Muller} and  \cite{GN}) and with a construction of primitive and imprimitive monodromy groups as treated in \cite{LZ}. Besides the facts mentioned  above,  
the problem is interesting in its own right. 
The present work is a contribuition to the study of this problem. Recently some contribuition        to this problem   was given
 in  \cite{BN0} and \cite{BN&GDL}. More specifically in \cite{BN&GDL} this problem has been considered for the   class of branched 
coverings $\phi:M \rightarrow N$ between 
connected closed surfaces. It consists in classifying 
which branch data (see definition below) can be realized by  indecomposable and which ones  can be realized   by decomposable. 
   In the work  \cite{BN&GDL}  this problem  is completely  solved in the case 
where $N$ is different from the sphere $S^2$ and the projective plane $\mathbb{R}P^{2}$. The main results  
are:

 \begin{theo}(\cite{BN&GDL}).
Every non-trivial admissible data are  realized on any  $N$ with $\chi(N)\leq 0$,  by an
indecomposable primitive branched covering. 
\end{theo}

\begin{Prop}(\cite{BN&GDL}). 
Admissible data $\mathscr{D}$  are decomposable on $N$,  with  $\chi(N)\leq 0$,
 if and only if there exists a factorization of $\mathscr{D}$
  such that its first factor is  non-trivial  admissible data.
\end{Prop}

 The purpose of this work is to study the same  question for branched coverings over  $N=\mathbb{R}P^{2}$ where the  
degree of the covering map $\phi: M \to \mathbb{R}P^{2}$ is even. 

 A branched covering $\phi: M \lra N$ of degree $d$ between closed connected surfaces determines  
a finite collection of partitions $\mathscr{D}$ of $d$, \emph{the branch
 data}, in correspondence with the {\it branch point set} $B_{\phi} \subset N$. The {\it total defect}
 of $\mathscr{D}$ is defined by $\nu(\mathscr{D})=\sum_{x \in B_{\phi}}d-\# \phi^{-1}(x)$. Conversely,
 given a collection of partitions $\mathscr{D}$ of $d$ and $N \neq S^{2}$, the necessary 
conditions for $\mathscr{D}$ to be a branch datum are sufficient to realize it. Let us point 
out that the  realization problem for $\mathbb{R}P^{2}$ has been solved in \cite{EKS}, in this 
case we will call $\mathscr{D}$ {\it admissible datum}.  

 From now on let $N=\mathbb{R}P^{2}$ and $\phi: M \to \mathbb{R}P^{2}$ a primitive branched covering, 
i.e. the induced homomorphism $\pi_1(M) \to \pi_1(\mathbb{R}P^{2})$ 
is surjective. Since $\phi$ is an
orientation-true map then $M$ is nonorientable, 
see \cite{GKZ}.  
If the homomorphism is not surjective the map admits an obvious decomposition.  

 The main result is:
  \begin{tma}
Let $d$ be even and $\mathscr{D}=\{D_{1},\dots,D_{s}\}$ an admisible datum. Then, $\mathscr{D}$ is realizable by an indecomposable
   branched covering if, and only if, either:\\
(1) $d=2$, or\\
 (2) There is  $i \in \{1,\dots,s\}$ such that
   $D_{i} \neq [2,\dots,2]$, or\\
(3) $d>4$ and $s>2$.
\end{tma}

 In Proposition 2.6\cite{BN&GDL} the authors classify admissible data realizable by decomposable primitive branched coverings, for $\chi(N) \leq 0$. Notice that the same proof applies when $\chi(N)=1$, i.e. $N=\mathbb{R}P^{2}$.
Then, except for the case $d=2$, where the branched covering clearly is never decomposable, we can assume in Theorem \ref{d par} that $\mathscr{D}$ is  realizable by decomposable primitive branched coverings. This is, if we join the main result in this paper with Proposition 2.6\cite{BN&GDL} we characterize admissible data realizable by both, decomposable and indecomposable primitive branched coverings.

 The case where $N=\mathbb{R}P^{2}$ and $d$ odd looks   more subtle and it is a work in progress.

The paper is divided into two sections apart from the introduction.  In Section 1, we quote the main
definitions and the results about realization of branched coverings over  $\mathbb{R}P^{2}$.  In Section 2, we characterize
branch data realizable by indecomposable
branched coverings with even degree.

\section{Preliminaries, terminology and notation}

\subsection{Permutation groups}
We denote by $\Sigma_{d}$ the symmetric group on a set $\Omega$ with $d$ elements and
 by $1_{d}$ its identity element. If $\alpha \in \Sigma_{d}$ and $x \in
 \Omega$, $x^{\alpha}$ is the image of $x$ by $\alpha$. An explicit permutation $\alpha$ will be written either 
as a product of disjoint cycles, i.e. its \emph{cyclic decomposition},  or in the following way:
 \begin{displaymath}
 \mathbf{\alpha=} 
\left( \begin{array}{ccccccc}
1 & 2& \dots & 2k+1\\ 
1^{\alpha} &2^{\alpha}& \dots& (2k+1)^{\alpha}
\end{array} \right),
\end{displaymath}   
  depending   on  what is more   convenient. The set of lengths of the
 cycles in the cyclic decomposition of $\alpha$, including the trivial ones, defines a partition of $d$, say
 $D_{\alpha}=[d_{\alpha_{1}},\dots,d_{\alpha_{t}}]$,  called 
 \emph{the cyclic structure of} $\alpha$. Define
 $\nu(\alpha):=\sum_{i=1}^{t}(d_{\alpha_{i}}-1)$, then $\alpha$ will be an \emph{even permutation} if $\nu(\alpha) \equiv 0 \pmod{2}$.
Given  a partition $D$ of $d$, we say  $\alpha \in
D$ if the cyclic structure of $\alpha$ is
$D$ and  we put $\nu(D):=\nu(\alpha)$.

For $1<r \leq d$, a permutation $\alpha \in \Sigma_{d}$ is called a
$r$-\emph{cycle} if in its cyclic decomposition its unique non-trivial cycle
has 
length $r$. Permutations $\alpha,\beta \in \Sigma_{d}$ are \emph{conjugate}
if there is $\lambda \in \Sigma_{d}$  such that
$\alpha^{\lambda}:=\lambda \alpha \lambda^{-1}=\beta$. It is a known fact that
conjugate permutations have the same cyclic structure.  

  Given a permutation group $G$ on $\Omega$  and $x \in \Omega$,  one  defines the {\it isotropy subgroup of $x$},
 $G_{x}:=\{g \in G: x^{g}=x\}$ , and the {\it orbit of $x$ by $G$}, $x^{G}:=\{x^{g}:g\in G\}$. For $H \subset G$, the subsets $Supp(H):=\{x
 \in \Omega: x^{h} \neq x \textrm{ for some $h \in H$}\}$ and
 $Fix(H):=\{x \in \Omega:x^{h}=x \textrm{ for all $h \in H$}\}$ are defined. For
 $\Lambda \subset \Omega$ and $g \in G$,
 $\Lambda^{g}:=\{y^{g}:y \in \Lambda\}$.
 
  The permutation  group $G$ is 
 \emph{transitive} if for all
 $x,y \in \Omega$ there is $g \in G$ such that
 $x^{g}= y$.
 A nonempty subset $\Lambda \subset \Omega$ is a \emph{block} of
 a transitive group $G$ if
 for each $g \in G$ either $\Lambda^{g} = \Lambda$ or
 $\Lambda^{g} \cap \Lambda =\emptyset$. A block $\Lambda$ is
 \emph{trivial} if either $\Lambda =\Omega$ or $\Lambda=\{ x \}$ for some $x \in
 \Omega$. Given a block $\Lambda$ of $G$, the set
 $\Gamma:=\{\Lambda^{\alpha}:\alpha \in G\}$ defines a partition of $\Omega$
 in blocks. This set is called \emph{a system of blocks
 containing} $\Lambda$ and the cardinality of $\Lambda$ divides the cardinality of $\Omega$. $G$ acts naturally on $\Gamma$.
 A transitive permutation group is
 \emph{primitive} 
 if it  admits  only  trivial blocks. Otherwise it is
 \emph{imprimitive}.
 
\begin{Ex}\label{e1}
A transitive permutation group $G<\Sigma_{d}$ containing a  $(d-1)$-cycle is 
  primitive.  Without loss of generality let us suppose that
 $g=(1\dots d-1)(d) \in G$. Then   
 any  proper subset $\Lambda$ of $\{1,\dots,d\}$ containing $d$ and at least one more element satisfies $\Lambda^{g}\neq \Lambda$ and
$\Lambda^{g}\cap \Lambda \neq \emptyset$. Thus the blocks
of $G$ are
trivial and $G$ is primitive.
\end{Ex}

\begin{Prop}[\cite{DM}, Cor. 1.5A]\label{dixon}
Let $G$ be a transitive permutation group on a set $\Omega$ with at least two
points. Then $G$ is primitive if and only if each isotropy subgroup $G_{x}$,
for $x \in \Omega$,
is a maximal subgroup of $G$. \qed
\end{Prop}

\subsection{Branched coverings on the projective plane}\label{pbc}
A surjective continuous open map $\phi:M \lra N$ between closed surfaces such 
that:
\begin{itemize}
\item for $x \in N$, $\phi^{-1}(x)$ is a totally disconnected set, and
\item there is  a non-empty discrete set $B_{\phi} \subset N$  such that the 
restriction $\hat{\phi}:=\phi|_{M-\phi^{-1}(B_{\phi})}$ is an ordinary
unbranched 
covering of degree $d$,
\end{itemize}
is called a \emph{branched covering of degree d over N}
 and it is  denoted  by
$(M,\phi,N,B_{\phi},d)$.  $N$ is \emph{the base surface}, $M$ is
\emph{the covering surface} and $B_{\phi}$ is \emph{the branch point set}.
 Its \emph{associated unbranched covering} is denoted by 
$(\widehat{M},\hat{\phi},\widehat{N},d)$, where
$\widehat{N}:=N-B_{\phi}$ and $\widehat{M}:=M-\phi^{-1}(B_{\phi})$. It is known that
$\chi(\widehat{M})=d\chi(\widehat{N})$, equivalently
\begin{eqnarray}\label{cE}
 \chi(M)-\#\phi^{-1}(B_{\phi})=d(\chi(N)-\#B_ {\phi}).
\end{eqnarray}
The set $B_{\phi}$ is  just
the image of the points in $M$ in which $\phi$ fails to be a local
homeomorphism. Then each $x \in B_{\phi}$ determines a non-trivial
partition $D_{x}$ of $d$, defined by the local degrees of $\phi$ on each component
in the preimage of a small disk $U_{x}$ around  $x$, with
$U_{x}\cap B_{\phi}=\{x\}$. The collection $\mathscr{D}:=\{D_{x}\}_{x \in B_{\phi}}$ is called
 \emph{the branch data} and its  \emph{total defect} is the positive integer
defined by 
$\nu({\mathscr{D}}):=\sum_{x \in B_{\phi}} \nu(D_{x})$. The total defect  satisfies
the \emph{Riemann-Hurwitz formula} (see \cite{EKS}): 
\begin{eqnarray}\label{rhf}
\nu(\mathscr{D})=d \chi(N)-\chi(M).
\end{eqnarray}
Associated to $(M,\phi,N,B_{\phi},d)$ we have a permutation group, {\it the monodromy group of $\phi$}, given by the
image of the \emph{Hurwitz's representation}
\begin{eqnarray}\label{Hr}
 \rho_{\phi}: \pi_{1}(N-B_{\phi},z) \lra \Sigma_{d},
 \end{eqnarray}
which  sends each class $\alpha \in \pi_{1}(N-B_{\phi},z)$ to a permutation of $\phi^{-1}(z)=\{z_{1},
\dots,z_{d}\}$, which indicates the terminal point of the lifting of a loop
in $\alpha$ after fixing the initial point. In particular,  for $x \in B_{\phi}$, let $c_{x}$
be a path
 from $z$ to a small circle $a_{x}$ about $x$ and define the loop class 
 $\mathbf{u}_{x}:=[c_{x}a_{x}c_{x}^{-1}]$. Then the cyclic structure of the permutation 
 $\alpha_{x}:=\rho_{\phi}(\mathbf{u}_{x})$ is given by $D_{x}$ and 
$\nu(\prod_{x \in B_{\phi}} \alpha_{x}) \equiv \nu(\mathscr{D}) \pmod{2}$.
The problem of realization of a branch data  is equivalent to an algebraic problem in term of representation on the symmetric group.
More precisely:
 \begin{theo}[See \cite{Hu}]
 Let $N$ be a surface,  $\mathscr{D}$ a finite collection of partitions of $d$ and $F\subset N$ such that $\# F= \#\mathscr{D}$. If it is possible to define a representation $\pi_{1}(N-F,z) \lra \Sigma_{d}$ like $\rho_{\phi}$, then $\mathscr{D}$ is realizable as branch datum of a branched covering on $N$.\qed
\end{theo}

\begin{Rm}
If $N=\mathbb{R}P^{2}$ and $\mathscr{D}=\{D_{1},\dots,D_{s}\}$,
to define $\rho_{\phi}$, it is necessary and sufficient to have  permutations $\alpha_{i} \in D_{i}$, for $i=1,\dots,s$, such that $\prod_{i=1}^{s}\alpha_{i}$ is a square. This follows  from the presentation 
$\pi_{1}(\mathbb{R}P^{2}-\{x_{1},\dots,x_{s}\})=\langle a, \mathbf{u_{1}},\dots,\mathbf{u_{s}} \vert  \prod_{i=1}^{s} \mathbf{u_{i}}=a^{-2}\rangle$.
\end{Rm}
\begin{Ex}\label{square}
If $r>0$ is an odd natural number then every $r-$cycle is the square of a permutation: 

if $\alpha=(a_{1}\;a_{2}\dots a_{r})$ then $\alpha=\beta^{2}$ where $\beta=(a_{1}\;a_{(\frac{r+1}{2})+1}\; a_{2}\; a_{(\frac{r+1}{2})+2}\dots a_{r}\; a_{\frac{r+1}{2}})$.
\end{Ex}

The brach data which can be realized by branched coverings over $\mathbb {R}P^2$    are given by: 
\begin{theo}[See \cite{EKS}] 
Let $\mathscr{D}$ be a collection of partitions of $d$. Then there is a branched covering $\phi:M\rightarrow \mathbb{R}P^{2}$ of degree $d$, with M connected and with branch data $\mathscr{D}$ if and only if 
\begin{eqnarray}\label{hcpp}
d-1\leq \nu(\mathscr{D})\equiv 0 \pmod{2}.
\end{eqnarray} Moreover, $M$ can be chosen to be nonorientable.
\qed
\end{theo}

\begin{Rm} The realization result above does not tell which  branch data can be realized by an orientable covering. In fact it is not hard to show that there is 
a bijection between the set of branched coverings over $\mathbb{R}P^2$  where the covering surface is orientable and the  set of branched  coverings over the sphere $S^2$  which have an even number of branched points. It is certainly an interesting problem to classify  such  realizable  brached data over the  sphere $S^2$ from the viewpoint of 
decomposibility.
\end{Rm}

\begin{definition}
A collection of partitions $\mathscr{D}$ of $d$ satisfying $(\ref{hcpp})$ will be called {\it admissible datum}.
\end{definition}

\section{Decomposability}
Given a covering, it is
\emph{decomposable} if it can be written as a composition of two non-trivial
coverings (i.e., both with degree bigger than 1), otherwise it is called
\emph{indecomposable}. In a decomposition of a branched
covering at least one of its components is a branched covering having proper branching.
 Moreover,
since the degree of a decomposable covering is the product of the degrees of
its components (see \cite{BBZ}, theorem 2.3), we are interested in branched
coverings with non-prime degree.

\begin{Prop}[See \cite{BN&GDL}]\label{yo}
A primitive branched covering
 is decomposable if and only if its monodromy
group is imprimitive. \qed
\end{Prop} 

Let $(M,\phi,\mathbb{R}P^{2},B_{\phi},d)$ be  primitive with branch data 
$\mathscr{D}=\{D_{1},\dots,D_{s}\}$, where $s:=\#B_{\phi}$. If $d$ is even, by 
(\ref{rhf}) and (\ref{hcpp}),  $\chi(M)$ is even and since $M$ is non-orientable
 $\chi(M)\leq 0$. Then by (\ref{cE}), $s>1$.
\begin{Prop}
Let $d>2$ be an even number and $\mathscr{D}=\{D_{1},\dots,D_{s}\}$ an admissible datum. If $\mathscr{D}$ contains a partition different of 
 $[2,\dots,2]$, then $\mathscr{D}$ can be realized by an indecomposable
 branched covering.
\end{Prop}

\begin{proof}
Without loss of generality, let us suppose
$D_{s} \neq [2,\dots,2]$.
Since $d \leq \nu(\mathscr{D}) \equiv 0 \pmod{2}$, there is $q \geq 0$ such that 
$\nu(\mathscr{D})=d+2q$ and
$\nu(D_{1})+ \nu(D_{2})=d+2q-\sum_{i=3}^{s}\nu(D_{i})$. If
$t:=\sum_{i=3}^{s}\nu(D_{i})-2q$ is bigger than zero, applying  Lemma 4.2
   \cite{EKS},   
there are permutations $\gamma_{1} \in D_{1}$, $\gamma_{2} \in D_{2}$
such that $\langle \gamma_{1}, \gamma_{2} \rangle$ acts in $\{1, \dots, d\}$
with  $t$ orbits and 
\begin{eqnarray}\label{t1}
\nu(\gamma_{1}\gamma_{2})=d-t=d-\sum_{i=3}^{s}\nu(D_{i})+2q.
\end{eqnarray}
If $t \leq 0$, define $r:=1-t$ and  applying
Lemma    4.3 \cite{EKS},
 for  $k:=-(1-\sum_{i=3}^{s}\nu(D_{i}))\equiv r \pmod{2}$, there exist $\gamma_{1}
\in D_{1}$, $\gamma_{2} \in D_{2}$ such that
$\langle \gamma_{1},\gamma_{2} \rangle$ acts trasitively on
$\{1,\dots,d\}$ and
\begin{eqnarray}\label{t2}
\nu(\gamma_{1}\gamma_{2})=(d-1)-k=d-\sum_{i=3}^{s}\nu(D_{i}).
\end{eqnarray}
Let $D_{12}$ be the partition determined by the cyclic structure of
 $\gamma_{1}\gamma_{2}$. Then (\ref{t1})
 implies 
$\nu(D_{12})+\nu(D_{3})=d+2q-\sum_{i=4}^{s}\nu(D_{i})$ and we can repeat the
analysis done before. On the other hand, the situation in 
(\ref{t2})     implies that $\nu(D_{12})+\nu(D_{3})=d-\sum_{i=4}^{s}\nu(D_{i})$ and
since $\sum_{i=4}^{s} \nu(D_{i})>0$, by Lemma    4.2 \cite{EKS}, there are
$\gamma_{12} \in D_{12}$, $\gamma_{3}\in D_{3}$ such that $\langle
\gamma_{12},\gamma_{3} \rangle$ acts with  $\sum_{i=4}^{s}\nu(D_{i})$ orbits
and 
$\nu(\gamma_{12}\gamma_{3})=d-\sum_{i=4}^{s}\nu(D_{i})$.
It is clear that repeating the analysis done before we will obtain one of the following
conditions: 
\begin{eqnarray}\label{t3}
 \nu(D_{12 \dots t-1})+ \nu(D_{s})=\left\{ \begin{array}{ll}
d+2q
& \textrm{ applying Lemma 4.2} \cite{EKS},
\\  
d
& \textrm{ applying Lemma 4.3} \cite{EKS},
\end{array}\right.
\end{eqnarray} 
where $D_{12 \dots j}$ denotes the partition determined by
 $\gamma_{1 \dots j-1}. \gamma_{j}$, where $\gamma_{1 \dots j-1} \in D_{1
  \dots 
j-1}$ and $\gamma_{j} \in D_{j}$ is obtained by succesive
applications of Lemmas 4.2\cite{EKS} and 4.3\cite{EKS}, for
$j=2,\dots,s-1$ .
Whichever the case, we are under the hypothesis of Lemma 4.5 \cite{EKS} then, since  
$D_{s} \neq [2, \dots,2]$, there are permutations $\gamma_{12 \dots s-1} \in 
D_{12 \dots s-1}$, $\gamma_{s} \in D_{s}$ such that the group  $\langle
\gamma_{12 
  \dots s-1}, \gamma_{s} \rangle$ acts transitively  on $\{1, \dots,d\}$ and
the product $\gamma_{12 \dots s-1}. \gamma_{s}$ is a $d-1$ cycle. Moreover by Example \ref{square}, there is a 
 permutation $\alpha \in \Sigma_{d}$ such that
$\gamma_{12 \dots 
  s-1}.\gamma_{s}=\alpha^{2}$ and the permutation group  $\langle \gamma_{12
  \dots s-1},\gamma_{s} \rangle$ is primitive by Example \ref{e1}. 
On the other hand, for $j=2,\dots,s-1$, there are 
$\lambda_{1},\dots,\lambda_{j} \in \Sigma_{d}$
such that
$\gamma_{12\dots j}=\gamma_{1}^{\lambda_{1}}\gamma_{2}^{\lambda_{2}} \dots
\gamma_{j}^{\lambda_{j}}$ (recall that
$\gamma_{j}^{\lambda_{j}}:=\lambda_{j} 
\gamma_{j} \lambda_{j}^{-1}$).
 Thus, we define the representation  
\begin{eqnarray*}
\rho: \langle
 a,\mathbf{u_{1}},...,\mathbf{u_{s}}|a^{2} \Pi_{i=1}^{s}\mathbf{u_{i}}=1] \rangle & \lra & \Sigma_{d} \\
 a & \longmapsto & \alpha^{-1},  \\
\mathbf{u_{i}}& \longmapsto & \gamma_{i}^{\lambda_{i}}, \\
\mathbf{u_{s}}& \longmapsto & \gamma_{s}.  
\end{eqnarray*}
Then, there exists a primitive branched covering 
 $(M,\phi, \mathbb{R}P^{2},B_{\phi},d)$ with $M$ nonorientable realizing
 $\mathscr{D}$ as branch data. Moreover, since  $\langle\gamma_{12 \dots s-1},
 \gamma_{ts} \rangle < 
 G:=\Im \rho$, then $G$ is a primitive permutation group and by Proposition \ref{yo},
  $(M,\phi,\mathbb{R}P^{2},B_{\phi},d)$ is indecomposable.
\end{proof}

\begin{Prop}\label{d>4t>2}
Let $d>4$ be even and $\mathscr{D}=\{D_{1},\dots,D_{s}\}$ an admissible datum such that $D_{i}=[2,\dots,2]$ for $i=1,\dots,s$. If $s>2$, there is an indecomposable branched 
covering  $(M,\phi,\mathbb{R}P^{2},B_{\phi},d)$ realizing $\mathscr{D}$.
\end{Prop}

\begin{proof}
Let $s>2$ be a natural number.  Since $\nu(D_{1})+\nu(D_{2})=d$,  by Lemma 4.5 \cite{EKS}
there are permutations $\gamma_{1} \in D_{1}$, $\gamma_{2} \in D_{2}$ such that
$\langle 
\gamma_{1},\gamma_{2} \rangle$ is transitive and $\gamma_{1}\gamma_{2} \in
D_{12}:=[d/2,d/2]$. Then
 $\nu(D_{12})+\nu(D_{3})=d-2+d/2$.
If $d/2$ is even, we apply again Lemma 4.5  \cite{EKS}        
 and we obtain permutations
$\gamma_{12} \in D_{12}$, $\gamma_{3} \in D_{3}$ such that $\langle
\gamma_{12},\gamma_{3}\rangle $ is transitive and $\gamma_{12}\gamma_{3} \in
D_{123}:=[d-1,1]$ is a
$(d-1)$-cycle (because $d/2 \neq 2$). Then by Example \ref{square}, there exist $\alpha \in
\Sigma_{d}$  such that
$\gamma_{12}\gamma_{3}=\alpha^{2}$ and $\langle \gamma_{12},\gamma_{3} \rangle$
is primitive. Since $\gamma_{12}$ and $\gamma_{1}\gamma_{2}$ are conjugates,
there is $\lambda \in \Sigma_{d}$ such that $\gamma_{12}=\lambda
\gamma_{1}\gamma_{2} \lambda^{-1}$. If $s$ is odd, we define the following
representation: 
\begin{eqnarray*}
\rho:  \langle a,\mathbf{\{u_{j}\}}_{j=1}^{s}| a^{2}\prod_{j=1}^{s}\mathbf{u_{j}}=1\rangle& \lra & \Sigma_{d} \\
 a & \longmapsto & \alpha^{-1},  \\
\mathbf{u_{1}}& \longmapsto & \lambda \gamma_{1} \lambda^{-1},\\
\mathbf{\{u_{i}\}}_{i=2}^{s-1}& \longmapsto & \lambda \gamma_{2} \lambda^{-1},\\
\mathbf{u_{s}}& \longmapsto & \gamma_{3}.  
\end{eqnarray*}
 If $s$ is even,  then
$\nu(D_{123})+\nu(D_{4})=d-2+d/2$ and again, applying  Lemma 4.5 \cite{EKS}
we obtain permutations $\gamma_{123} \in D_{123}$ and $\gamma_{4} \in D_{4}$ such
that 
$\langle \gamma_{123},\gamma_{4} \rangle$ is transitive and
$\gamma_{123}\gamma_{4}$ is a $(d-1)$-cycle.  Then by Example \ref{square}, there is $\alpha \in \Sigma
_{d}$ such that $\gamma_{123}\gamma_{4}=\alpha^{2}$ and $\langle \gamma_{123},
\gamma_{4} \rangle$ is primitive by Example \ref{e1}. Notice that there exist
$\lambda_{1},\lambda_{2},\lambda_{3} \in \Sigma_{d}$ such that
$\gamma_{123}=\gamma_{1}^{\lambda_{1}}\gamma_{2}^{\lambda_{2}}\gamma_{3}^{\lambda_{3}}$
 and in this case we define the following representation:
\begin{eqnarray*}
\rho:  \langle a,\mathbf{\{u_{j}\}}_{j=1}^{s}|
 a^{2}\prod_{j=1}^{s}\mathbf{u_{j}}=1\rangle& \lra & \Sigma_{d} \\ 
 a & \longmapsto & \alpha^{-1}, \\
\mathbf{u_{1}}& \longmapsto & \gamma_{1}^{\lambda_{1}}, \\
\mathbf{\{u_{i}\}}_{i=2}^{s-2}& \longmapsto & \gamma_{2}^{\lambda_{2}},\\
\mathbf{u_{s-1}}& \longmapsto & \gamma_{3}^{\lambda_{3}},\\
\mathbf{u_{s}}& \longmapsto & \gamma_{4}. 
\end{eqnarray*}
Whichever the case $G:=Im (\rho)$ is primitive by
Proposition \ref{yo}, the branched covering associated to $G$ is
indecomposable. 

If $d/2$ is odd, since $\nu(\mathscr{D})=sd/2$, the hypothesis implies  $s\geq 4$ even.  Since
$\nu(D_{12})+ \nu(D_{3})=(d-1)+(d/2-1)$,  by Lemma 4.3 \cite{EKS}
there are $\gamma_{12} \in D_{12}$, $\gamma_{3} \in D_{3}$ such that $\langle
\gamma_{12},\gamma_{3} \rangle$ is transitive and $\gamma_{12}\gamma_{3}$ is a 
$d$-cycle. Let $D_{123}:=[d]$ thus $\nu(D_{123})+\nu(D_{4})=d+(d/2-1)$ and by Lemma 4.5 in \cite{EKS}, we obtain permutations $\gamma_{123} \in
D_{123}$, $\gamma_{4} \in D_{4}$ such that $\langle \gamma_{123}, \gamma_{4}
\rangle$ is transitive and $\gamma_{123}\gamma_{4}$ is a $(d-1)$-cycle. Then by Example \ref{square},
there is $\alpha \in \Sigma_{d}$ such that $\gamma_{123} \gamma_{4}=
\alpha^{2}$ and, by Example \ref{e1}, 
$\langle \gamma_{123}, \gamma_{4} \rangle$ is primitive. For this case we
define a representation like the last in the case before.
\end{proof}

\begin{Lemma}\label{unico}
Let $d\neq 2$ be even and $\alpha, \beta \in [2,\dots,2]$, such that $G:=\langle \alpha,
\beta \rangle$ is transitive. Then $G$ is imprimitive and unique up to
conjugation. 
\end{Lemma}

\begin{proof}
Let us suppose $d=2k$, with $1 < k \in \mathbb{N}$. To obtain 
$\langle \alpha,\beta \rangle$ transitive, it is 
necessary that transpositions in $\beta$ ``link" $k$ transpositions
 in $\alpha$. For that, we need $k-1$ transpositions and thus $\beta$ is automatically
 defined.
Then, up to conjugation, we can consider
$\alpha =(1\;2)(3\;4)\dots(d-1 \; d)$ and $\beta=(2 \;3)(4 \; 5)\dots (d-2 \;
d-1)(d \; 1)$. Thus $\alpha \beta=(1\;3\dots d-1)(2\;4\dots d)$ and  the set
 $B=\{1,3, \dots, d-1\}$ will be a nontrivial block of $G=\langle \alpha,
 \beta \rangle$, which makes it imprimitive. 
\end{proof}

\begin{Prop}\label{d>4t=2}
A primitive branched covering realizing
$\{[2,\dots,2],[2,\dots,2]\}$ is decomposable. 
\end{Prop}

\begin{proof}
Let $(M,\phi,\mathbb{R}P^{2},\{x,y\},d)$ be a primitive branched covering with
branch data $\{[2,\dots,2],[2,\dots,2]\}$. If 
\begin{eqnarray*}
\rho: \langle a, \mathbf{u_{1}}, \mathbf{u_{2}} | a^{2} \mathbf{u_{1}u_{2}} =1 \rangle & \lra &
\Sigma_{d}\\
a & \longmapsto & \alpha, \\
\mathbf{u_{1}} & \longmapsto & \gamma_{1},\\
\mathbf{u_{2}}& \longmapsto & \gamma_{2},
\end{eqnarray*}
is its Hurwitz's representation, 
$G:=\Im \rho=\langle \alpha, \gamma_{1}, \gamma_{2} \vert \gamma_{1}^{2}=
\gamma_{2}^{2}=1, \alpha^{2}\gamma_{1}\gamma_{2}=1 \rangle$ is a transitive permutation
 group with  $\gamma_{1}, \gamma_{2} \in 
[2,\dots,2]$.

If $\langle \gamma_{1},\gamma_{2} \rangle$ is transitive, by Lemma \ref{unico}
it is  imprimitive and the branched covering is decomposable. If not, using
relations in $G$, is easy to see that  
$[Fix (\gamma_{1}\gamma_{2})]^{\gamma_{i}} \subset Fix (\gamma_{1}\gamma_{2})$, for $i=1,2$, 
and $[Fix (\gamma_{1}\gamma_{2})]^{\alpha} \subset Fix (\gamma_{1}\gamma_{2})$. Then
\begin{eqnarray*}
Fix(\gamma_{1}\gamma_{2})=[Fix(\gamma_{1}\gamma_{2})]^{\gamma_{i}}=
[Fix(\gamma_{1}\gamma_{2})]^{\alpha},
\end{eqnarray*}
because $\gamma_{i}$, for $i=1,2$, and $\alpha$ are permutations.
  Then
for all $g \in G$  we have 
$[Fix(\gamma_{1}\gamma_{2})]^{g}=Fix(\gamma_{1}\gamma_{2})$. If 
$Fix (\gamma_{1}\gamma_{2})\neq \emptyset$, since $G$ is transitive, 
then 
$\gamma_{1}\gamma_{2}=1$. Up to conjugation,
$\gamma_{1}=\gamma_{2}=(1 \; 2)(3 \; 
4)\dots(d-3 \; d-2)(d-1 \; d)$. By the relation, the options for  
$\alpha$ are either $\alpha:=(2 \; 3)(4 \; 5)\dots(d-2 \;
d-1)(d \quad 1)$  or $\alpha:=(1)(2 \; 3)(4 
\; 5)\dots(d-2 \; d-1)(d)$. The first option implies $G$ equal
 to the group in Lemma \ref{unico}, therefore it is
imprimitive. The second one implies  $\{1,d\}$ as a block. If 
$Fix (\gamma_{1}\gamma_{2})=\emptyset$ then
every cycle of $\alpha$ has
length $\geq3$ and  $\gamma_{1}$,  
 $\gamma_{2}$ have not common cycles. Let
$O_{1}, \dots, O_{k}$ be the orbits of the action of $\langle
\gamma_{1},\gamma_{2} 
\rangle$  
on $\{1, \dots, d\}$, with  $k>1$. Notice that $\#O_{i} \geq 4$ is even,
because each transposition of $\gamma_{2}$ linked transpositions of 
$\gamma_{1}$ then, it connects an even number of elements. On the other hand, if
 $\langle \gamma_{1},\gamma_{2}\rangle_{i}$ denotes the restriction of 
  $\langle \gamma_{1},\gamma_{2}\rangle$ on  $O_{i}$, we are in the
 situation of  Lemma \ref{unico}, therefore $\langle
 \gamma_{1},\gamma_{2}\rangle_{i}$ is imprimitive. If $\#O_{i}=2n$ for 
$n \in \mathbb{Z}^{+}$, then its elements appear in $\gamma_{1}\gamma_{2}$ in the form
$(a_{i_{1}} \dots a_{i_{n}})(a_{i_{n+1}} \dots a_{i_{2n}})$. 
Considering  the relation $\gamma_{1}\gamma_{2}=\alpha^{-2}$,
 we conclude that $\alpha$ will
connect two orbits, $O_{i}$ and $O_{j}$, only if $\#O_{i}=\#O_{j}$. But $\alpha$ 
makes the group $G$  transitive, and so  all orbits have the same cardinality equal 
to $2n$. For example, if $i \neq j$ and the elements of $O_{j}$ are in 
 $\gamma_{1}\gamma_{2}$
in the form $(b_{j_{1}} \dots
b_{j_{n}})(b_{j_{n+1}} \dots b_{j_{2n}})$, whitout loss of
generality 
$(a_{i_{1}}\;b_{j_{1}}\;a_{i_{2}}\;b_{j_{2}}\dots 
a_{i_{n}}\;b_{j_{n}})$ is a cycle of $\alpha$ and
thus, the blocks of $\langle \gamma_{1},\gamma_{2} \rangle_{i}$, for $i=1,\dots,k$,
 become blocks for $G$. 
Then $G$ is imprimitive and the branched covering is decomposable. 
\end{proof}

\begin{Prop}
A primitive branched covering of degree $4$ realizing the finite
collection  $\mathscr{D}=\{[2,2],\dots, 
[2,2]\}$ is decomposable.
\end{Prop}

\begin{proof}
Let $(M,\phi,\mathbb{R}P^{2},B_{\phi},4)$ be a primitive branched covering with branch
data  $\mathscr{D}$. 
Suppose
$\nu(\mathscr{D})=2t$, $2 \leq t \in \mathbb{Z}^{+}$. Let
 \begin{eqnarray*}
\rho: \langle a, \mathbf{u_{1}}, \dots, \mathbf{u_{t}} | a^{2} \prod_{i=1}^{t}\mathbf{ u_{i}} =1 \rangle & \lra &
\Sigma_{4}\\
a & \longmapsto & \alpha \\
\mathbf{u_{i}} & \longmapsto & \gamma_{i}.
\end{eqnarray*}
be its Hurwitz's representation. Note that 
the possible images for
$\prod_{i=1}^{t} \mathbf{u_{i}}$ are, without loss of generality, either $(1)(2)(3)(4)$ or
$(1 2)(3 4)$.
Define $U:=\langle \gamma_{1}, \dots, \gamma_{t} \rangle$.
If $U$ is transitive,
then $U \cong \langle (12)(34),(13)(24) \rangle$ is imprimitive, because each
pair of elements is a block. Thus, if $\rho(\prod_{i=1}^{t}\mathbf{u_{i}})=1$, 
the group
$G:=\Im \rho=\langle U,\alpha \vert \alpha^{2}\prod_{i=1}^{t}=1 \gamma_{i} \rangle$
 is imprimitive, for all $\alpha$. On the other
hand, if
 $\rho(\prod_{i=1}^{t}u_{i})=(12)(34)$ then, either
$\alpha=(1324)$ or $\alpha=(1423)$. Whichever the case, $\{1,2\}$
is a block.
If $U$ is not transitive, then  $U \cong \langle (12)(34) \rangle$ and, for guarantee
 the transitivity of  $G$, we have
$\alpha=(1324)$. Thus $\{1,2\}$ is a block and $G$ is imprimitive.
\end{proof}

We summarize the case $d$ even in the following theorem:

\begin{theo}\label{d par}
Let $d$ be even and $\mathscr{D}=\{D_{1},\dots,D_{s}\}$ an admmisible datum. Then, $\mathscr{D}$ is realizable by an indecomposable
   branched covering if, and only if,  either:\\
(1) $d=2$, or\\
 (2) There is  $i \in \{1,\dots,s\}$ such that
   $D_{i} \neq [2,\dots,2]$, or\\
(3) $d>4$ and $s>2$.\qed
\end{theo}


Departamento de Matem\'atica DM-UFSCar

Universidade Federal  de S\~ao Carlos

Rod. Washington Luis, Km. 235. C.P 676 - 13565-905 S\~ao Carlos, SP - Brasil

nbedoya@dm.ufscar.br

\vspace{0.5cm}
Departamento de Matem\'atica

Instituto de Matem\'atica e Estat\'istica

Universidade de S\~ao Paulo

Rua do Mat\~ao 1010, CEP 05508-090, S\~ao Paulo, SP, Brasil.

dlgoncal@ime.usp.br

\end{document}